\documentclass[leqno]{amsart}
\usepackage[T1,T2A]{fontenc}
\usepackage[utf8]{inputenc}
\usepackage[english]{babel}
\usepackage{amsthm}
\usepackage{color}
\usepackage{amssymb}
\usepackage{amsmath,amsfonts,enumerate}
\usepackage{amsthm}
\usepackage{amsmath}
\usepackage{graphicx}
\usepackage{hyperref}

 \newtheorem{theorem}{Theorem}[section]
      \newtheorem{lemma}[theorem]{Lemma}

      \newtheorem{definition}[theorem]{Definition}

      \newtheorem{remark}[theorem]{Remark}

\usepackage{color}
\newcommand{\be}{\begin{eqnarray*}}
\newcommand{\ee}{\end{eqnarray*}}
\newcommand{\beq}{\begin{equation}}
\newcommand{\eeq}{\end{equation}}
\newcommand{\beqn}{\begin{equation*}}
\newcommand{\eeqn}{\end{equation*}}
\newcommand{\bsp}{\begin{split}}
\newcommand{\esp}{\end{split}}
\DeclareMathOperator\supp{supp}

\begin{document}

\title{A note on weak compactness criteria in $L_1$}

\author{Y. Nessipbayev}
\address{
School of Mathematics and Statistics, University of New South Wales, Kensington, NSW, 2052, Australia;
Institute of Mathematics and Mathematical Modeling, 050010 Almaty, Kazakhstan;}
\email{y.nessipbayev@unsw.edu.au}

%\author{F. Sukochev}
%\address{School of Mathematics and Statistics, University of New South Wales, Kensington, NSW, 2052, Australia}
%\email{f.sukochev@unsw.edu.au}

\author{K. Tulenov}
\address{
Al-Farabi Kazakh National University, 050040 Almaty, Kazakhstan;
Institute of Mathematics and Mathematical Modeling, 050010 Almaty, Kazakhstan; Department of Mathematics: Analysis, Logic and Discrete Mathematics, Ghent University, Ghent,
Belgium.}
\email{tulenov@math.kz}

\date{}
\begin{abstract}
We provide a direct proof for the equivalence of K.M. Chong's and De la Vall\'{e}e Poussin's criteria of weak compactness for a subset $K$ of $L_1(0,1)$. Furthermore, we prove the equivalence in $L_1(0, \infty)$ under some additional condition.
\end{abstract}

\maketitle

\section{Introduction}
It is well known that for a subset $K\subset L_1(\Omega, \Sigma, \nu)$, where $(\Omega, \Sigma, \nu)$ is a finite measure space, the following conditions are equivalent:
%\begin{enumerate}[{\rm (i)}]
%         \item $\varphi(0)=0;$
%          \item $\varphi(t)$ is positive and increasing for $t>0;$
%         \item $\frac{\varphi(t)}{t}$ is decreasing for $t>0.$
%\end{enumerate}
\begin{enumerate}[\rm(i)]
  \item  $K$ is a relatively weakly compact set;
  \item\label{ii}  $K$ is bounded and uniformly integrable (Dunford-Pettis criterion, see \cite[Theorem 15, p.76]{Diestel}, \cite{D-P}, \cite[Theorem 23, p.20]{Meyer});
  \item\label{iii} there exists an $N$-function $F$ (see Definitions \ref{N} and \ref{O}) such that
$$\sup \left\{\int F(f)d\nu: \ f \in K \right\}< \infty$$
(De la Vall\'{e}e Poussin's criterion, see \cite[Theorem 22, p.19-20]{Meyer}, see also \cite[Theorem 2, p.3]{RaoRen});
  \item $K$ is contained in the orbit (see formula \ref{orbit}) of some positive integrable function (in the sense of the Hardy-Littlewood-P\'{o}lya submajorization) (K.M. Chongs's criterion, see  \cite[Theorem  4.2]{Chong}).
\end{enumerate}

Regarding the condition \eqref{ii}, we recall that a subset $K$ of $L_1(\Omega, \Sigma, \nu)$ is called uniformly integrable if for each $\varepsilon>0$ there is $\delta>0$ such that
$$\int_E |f|d\mu<\varepsilon$$
whenever $f\in K$ and $\mu(E)<\delta$.

The concept of uniform integrability can be easily generalized to any Banach lattice $X$ of measurable functions over a measure space $(\Omega, \Sigma, \nu)$. We shall say that a set $K \subset X$ has equi-absolutely continuous norms in $X$ if (see, e.g. \cite{As})
$$\lim_{\delta \to 0} \sup_{\nu(E)< \delta} \sup_{x \in K} ||x \chi_E||_X=0.$$

As for the condition \eqref{iii}, the study of weak compactness in Orlicz spaces ($L_1$ itself is an example of an Orlicz space) was of interest to W. Orlicz himself, who proved that each Orlicz space $L_G = L_G (0, 1)$ such that
$$\lim_{t \to \infty} \frac{G^*(2t)}{G^*(t)}=\infty,$$
where $G^*$ is the complementary (see
\cite[Chapter 1, formula (2.9)]{KR}) function to an Orlicz function $G$ (see Definition \ref{O}), satisfies Dunford-Pettis criterion of weak compactness \cite[assertion 1.5]{Orlicz} (see also \cite{Alex}), that is, every relatively weakly compact subset of
$L_G$ has equi-absolutely continuous norms in $L_G$.  
%We note that defintions of Orlicz functions and $N$-functions almost coincide (see Definitions \ref{N}, \ref{O}).

As noted in \cite[Ch.1]{RaoRen} \lq\lq The uniform integrability concept through its equivalence with a
condition discovered by De la  Vall\'{e}e Poussin in 1915 has given a powerful inducement for the study of Young's functions (or $N$-functions) and the
corresponding function spaces.\rq\rq

%There is a substantial literature devoted to the study of weak compactness in both Orlicz function and sequence spaces, see, for example \cite{Alex, Ando, As, Barcenas, Chong2, Chong, DSS, KM, Lefevre, Nowak, RaoRen, Zhang, Zhang2}, and references therein.

These characterisations of weak compactness have been shown time and time again to be powerful tools in functional analysis, and have served as sources of inspiration for much subsequent research  (see \cite{Alex, Ando, As, Barcenas, Chong2, Chong, DSS, KM, Lefevre, LMJ, Nowak, RaoRen, SSC, Zhang, Zhang2}).

Our main result of Section \ref{section finite} (see Theorem \ref{ChVP}) is of mostly pedagogical value: we prove directly the equivalence of K.M. Chong's (condition (iii)) and De la Vall\'{e}e Poussin's (condition (iv)) criteria of weak compactness of a subset $K$ of the space $L_1(0,1)$. In fact, this equivalence is obvious, as both criteria are known to be equivalent to relative weak compactness of $K$ in $L_1(\Omega, \Sigma, \nu)$, however, our proof does not refer to Dunford--Pettis criterion and provides a clear demonstration of powerful methods from the general theory of symmetric function spaces. We also prove that any function from $L_{1}(0, \infty)$ belongs to some Orlicz space, different from $L_1(0, \infty)$ itself (see Lemma \ref{N-function}). For the case of an integrable function from a finite measure space, this result is known (see \cite[Chapter II, p.60]{KR}). 
%In fact, the proof is similar to the case of finite measure space, however, we need to choose another partition of $(0, \infty)$, different from the partition of $(0,1)$ in \cite[Chapter II, p.60]{KR}. 

In Section \ref{section semifinite} we discuss the equivalence in the setting of a $\sigma$-finite measure space. Using Lemma \ref{N-function} it is straightforward to show that the Chong's condition implies the condition of  De la Vall\'{e}e Poussin in $L_1(0, \infty)$. We also show that the converse statement is also true under some additional condition (see Theorem \ref{expl of th}).

\section{Preliminaries}

%Recall that a subset $K$ of a space $L_1(\nu)$ is called uniformly integrable if, for any $\varepsilon>0$, there exists $\delta>0$ such that $\sup\left\{\int_{E}|f|d\nu: \ f \in K\right\}< \varepsilon$  whenever $\nu(E)<\delta.$ In particular, every bounded subset of $L_2$ is uniformly integrable. Alternatively, $K$ is bounded and uniformly integrable if and only if, for any $\varepsilon>0$, there is $N>0$ such that $\sup\left\{\int_{|f|>c}|f|d\nu: \ f \in K\right\}< \varepsilon$ whenever $c \geq N$ (see \cite[p.2]{Alex}).

Let $(I,m)$ denote the measure space, where $I = (0,\infty)$ (resp. $(0,1)$), equipped with the Lebesgue measure $m.$
 Let $L(I,m)$ be the space of all measurable real-valued functions on $I$ equipped with the Lebesgue measure $m$. Define $S(I,m)$ to be the subset of $L(I,m)$, which consists of all functions $f$ such that $m(\{|f| > s\}) < \infty$ for some $s > 0.$ Note that if $I=(0,1)$, then $S(I,m)=L(I,m).$

For $f\in S(I,m)$, we denote by $\mu(f)$ the decreasing rearrangement of the function $|f|.$ That is,
$$\mu(t,f)=\inf\{s\geq0:\ m(\{|f|>s\})\leq t\},\quad t \ge0.$$

We say that $f$ is submajorized by $g$ (both in $S(I,m)$) in the sense of Hardy--Littlewood--P\'{o}lya (written $f\prec\prec g$) if
$$\int_0^t\mu(s,f)ds\leq\int_0^t\mu(s,g)ds,\quad t\geq0.$$

%Also, we say that $f$ is majorized by $g$ on $I$ in the sense of Hardy--Littlewood--P\'{o}lya (written $f\prec g$) if in addition to $f\prec\prec g$, we have
%$$\int_{I}\mu(s,f)ds=\int_{I}\mu(s,g)ds.$$

For a positive function $g \in L_1(I,m)$ we define the following set
\begin{equation}\label{orbit}
\mathcal{C}_g:= \{f\in L_1(I,m): \quad |f| \prec \prec g \},
\end{equation}
which is called the orbit of a function $g$.

\subsection{Marcinkiewicz spaces}\label{Marc}

Let  $\psi:[0,\infty)\rightarrow[0,\infty)$ be an increasing concave function such that
$\psi(0+)=0.$
For any such function $\psi$ the Marcinkiewicz space $M_{\psi}(I)$ is defined by setting

$$M_{\psi}(I)=\{f \in S(I): \|f\|_{M_{\psi}(I)} < \infty \},$$
equipped with the norm
$$\|f\|_{M_{\psi}(I)}=\sup_{t\in I} \frac{1}{\psi(t)} \cdot \int_{0}^{t} \mu(s,f)ds.$$
For more details on Marcinkiewicz spaces of functions, we refer the reader to \cite[Chapter 2.5]{BSh} and \cite[Chapter II.5]{KPS}.

\subsection{Orlicz spaces}

\begin{definition}\label{N} 
An $N$-function is a continuous, convex function $G: [0, \infty) \rightarrow [0, \infty)$ satisfying the following properties (cf. \cite[Proposition 1.1]{Alex}):
\begin{enumerate}[{\rm (i)}]
\item $G(0)=0$,
\item $G(\lambda)>0$ for $\lambda>0$,
\item $\frac{G(\lambda)}{\lambda} \rightarrow 0$ as $\lambda \rightarrow 0+$,
\item $\frac{G(\lambda)}{\lambda} \rightarrow \infty$ as $\lambda \rightarrow \infty$.
\end{enumerate}
\end{definition}

For our purposes, the behavior of an $N$-function $G$ for small arguments (condition (iii)) is not crucial.

\begin{definition}\label{O} 
An Orlicz function is a function $G: [0, \infty) \to [0, \infty]$ with the following properties (cf. \cite[p.258]{KM}):
\begin{enumerate}[{\rm (i)}]
\item $G(0)=0$,
\item $G$ is not identically equal to zero,
\item $G$ is convex,
\item $G$ is continuous at zero.
\end{enumerate}
\end{definition}

It is worth noting that while every $N$-function is also an Orlicz function, the converse is not always true. Hereafter, unless stated otherwise, we denote by $G$ an $N$-function. For such a function, we consider the (extended) real-valued functional $\mathbf{G} (f)$, also referred to as the modular defined by an $N$-function $G$. This functional is defined on the class of all measurable functions $f$ on $I$ by
$$\mathbf{G} (f)=\int_{I} G(|f(t)|)dt.$$

The set
$$L_{G}=\{f \in S(I,m): \quad \|f\|_{L_G}<\infty\},$$
where
$$\|f\|_{L_G}=\inf\left\{c>0: \quad \int_{I} G\left(\frac{|f|}{c}\right)dm \leq 1\right\},$$
is said to be an Orlicz space defined by the Orlicz function $G$.

We denote by $G^*$ the function complementary (or conjugate) to $G$ in the sense of Young, defined as (cf. \cite[Chapter 1, p.11]{KR})

$$G^*(t)=\sup\{s|t| - G(s): \,\,\ s \geq 0\}.$$

It's noteworthy that $G^*$ is an Orlicz function (cf. \cite[p.258]{KM}).

\begin{definition}
A function
$$\varphi_{L_G}(x)=\|{\bf{1}}_{[0,x]}\|_{L_G}, \quad x \ge 0$$
is called the fundamental function of the Orlicz space $L_G$ ($\bf{1}$ stands for the characteristic function).
\end{definition}

\section{Equivalence of Chong's and De la Vall\'{e}e Poussin's criteria in $L_1(0,1)$}\label{section finite}
In this section, we discuss the
equivalence of K.M. Chong's and De la Vall\'{e}e Poussin's criteria of relative weak compactness of a subset $K\subset L_1(0,1)$ in terms of an Orlicz function $G$.
The following result presents an extension of \cite[Chapter II, p.60]{KR} to $\sigma$-finite measure spaces.
\begin{lemma}\label{N-function}
For any integrable function $f$ on $I=(0, \infty)$, there exists an $N$-function $G$ such that $G(|f|)$ is integrable on $I$. %Moreover, $\frac{G(t)}{t} \rightarrow \infty$ as $t \rightarrow \infty.$
\end{lemma}

\begin{proof}
Note that if $f=0$ on a set $I$, then $G(f)\equiv0$ on $I$. Hence $\int_{I}G(f(t))dt=0$, so $G(f)$ is integrable on $I$. We set
$$\supp f=\{t \in [0, \infty): f(t)\neq0\}.$$
Consider the family of pairwise disjoint sets
$$I_{n} = \{t\in \supp f: 2^{n} \leq |f(t)| <2^{n+1}\}, \,\,\,\ n\in \mathbb{Z}.$$ Then $(0, \infty)=I \supseteq \bigcup_{n =-\infty}^{\infty} I_n$, and $f$ is integrable on $I_n \ \text{for all} \  n\in \mathbb{Z}.$ Hence,
$$\sum_{n=-\infty}^{\infty} 2^{n}\cdot m(I_n) \leq \int_{0}^{\infty} |f(t)|dt <\infty.$$
By Lemma \ref{sequence} in the Appendix below there exists an increasing sequence of real numbers $\{\alpha_n \}_{n=-\infty}^{\infty}$ with $\alpha_n=0$ for all $n \leq 0$ such that
$\lim_{n\to\infty} \alpha_n=\infty$ and

\begin{equation}\label{alpha}
\sum_{n=-\infty}^{\infty} \alpha_{n+1} \cdot 2^{n} \cdot m(I_n) < \infty.
\end{equation}
We set

$$
p(t) =
\begin{cases}
t &\text{if \;\;  $ 0\leq t <1$},
\\
\alpha_n &\text{if \;\; $2^{n-1}\leq t < 2^{n}$}  \;\; (n=1,2,...).
\end{cases}
$$

Without loss of generality we may assume $\alpha_1 \geq 1$. Since $p(t)$ is nondecreasing and right-continuous, $p(0)=0,$ $p(t)>0$ whenever $t>0,$ and $\lim_{t\rightarrow \infty} p(t) = \infty$ we may define an $N-$function $G$ (see \cite[Definition 1.1, p.3]{Alex})  by

$$G(x) = \int_{0}^{x} p(t)dt, \ \ \ \ x \geq 0.$$
Since
$$G(2^n) =  \int_{0}^{2^n} p(t)dt \leq  \int_{0}^{2^n} \alpha_n dt = 2^n \cdot \alpha_n, \ \ \ \ n=1,2,... ,$$ it follows, in virtue of \eqref{alpha}, that

\begin{eqnarray*}\begin{split}
\int_{0}^{\infty} G(|f(t)|)dt & = \sum_{n=-\infty}^{\infty} \int_{I_n} G(|f(t)|)dt \\
&\leq \sum_{n=-\infty}^{\infty} G (2^{n+1}) m(I_n) \leq \sum_{n=-\infty}^{\infty}  2^{n+1} \cdot \alpha_{n+1} \cdot m(I_n)<\infty.\
\end{split}
\end{eqnarray*}
Hence, $G(|f|)$ is integrable on $(0, \infty)$.
The condition $\frac{G(t)}{t}=\frac{\int_{0}^{t}p(s)ds}{t} \rightarrow \infty$ as $t \rightarrow \infty$ follows immediately by applying the L'H\^{o}pital's rule.
\end{proof}

\begin{remark}
Observe, that if we had asked in Lemma \ref{N-function} for an Orlicz function $G$ (instead of an $N$-function), then there would be nothing to prove as one may choose
 $G(t)\equiv t$ for all $t \in [0, \infty).$
\end{remark}

Recall, in \cite[Lemma 4.1]{Chong} K.M. Chong proved that a weakly compact set in $L_1$ associated with a finite measure space is a subset of the orbit of some positive integrable function.

Another characterization of uniform integrability (or relative weak compactness) is given in a theorem of De la Vall\'{e}e Poussin \cite[Theorem 22, p.19-20]{Meyer}, which states the following:
{\it
A subset $K$ of $L_1(I,m)$ (with $m(I)<\infty$) is bounded and uniformly integrable if and only if there is an Orlicz function $G$ such that
$\frac{G(t)}{t} \to \infty$ as $t \to \infty$ so that $$\sup\left\{ \int_{I} G(|f|)dm: f\in K \right\}<\infty.$$}

\begin{remark}
In the theorem of De la Vall\'{e}e Poussin above, we may omit the boundedness condition as uniform integrability trivially implies boundedness.
\end{remark}

The following lemma may be found in \cite[Proposition 2.4]{HSZ} for a finite measure space (see also \cite[p. 22, Theorem D.2]{MOA} and  \cite{Weyl}), or in \cite[Proposition 2.3]{Hiai},  \cite[Proposition 1.2]{HN2} for an infinite measure space.

\begin{lemma}\label{Prop:lefunction2}
Assume that  $f=\mu(f) $ and $g=\mu(g)$ are integrable functions  on $(0,\infty)$.
If $\int_0^t  f(s)  ds \le  \int_0^t g(s)ds $ for every $0 <t< \infty $,
then for every increasing continuous convex function $\varphi$ on $(0,\infty)$, we have $\int_0^t \varphi(f(s)) ds \le  \int_0^t  \varphi(g(s ))ds $ for every $0 <t < \infty$.
\end{lemma}

The following theorem, the main result of this section,  provides the direct proof for the equivalence of K.M. Chong's and De la Vall\'{e}e Poussin's criteria of weak compactness of a subset $K$ of $L_1(0,1)$.

\begin{theorem}\label{ChVP}
Let $K$ be a bounded subset of $L_1(0,1)$, then the following two conditions are equivalent:
\begin{enumerate}[{\rm (i)}]
\item\label{a} there exists an Orlicz function $G$ with $\frac{G(t)}{t} \to \infty$ as $t \to \infty$ so that $$\sup \left\{ \int_{0}^1 G(|f|)ds: \ f\in K \right\}<\infty;$$
\item\label{b} there exists a positive function $g\in L_1 (0,1)$ such that $|f| \prec\prec g$ for all $f \in K,$ that is, $K$ is contained in the ordit of $g.$
\end{enumerate}
\end{theorem}

\begin{proof}  $ \eqref{a} \Longrightarrow \eqref{b}$.
Suppose \eqref{a} holds. Hence, $K$ is a bounded subset of $L_G=L_G(0,1).$ Without loss of generality, we may assume that $K$ is in the (closed) unit ball of  $L_G.$

Let $\varphi$ be a fundamental function of $L_G$. The function $\varphi$ is quasiconcave. Let $\psi$ be its least concave majorant, so $\frac{1}{2} \psi \leq \varphi \leq \psi$ (see e.g. \cite[p. 71, Proposition 5.10]{BSh}). The Marcinkiewicz space  $M_\psi$ contains the Orlicz space $L_G$ (see \cite[Theorem II. 5.13, p.72]{BSh}, see also \cite[Corollary II. 5.14, p.73]{BSh}). By Theorem II.5.7 from \cite{KPS} we know that
 $K$ lies in a unit ball of $M_\psi.$
Hence, by (2.12) in \cite[p.64]{KPS}, we have
$$\int_{0}^{t} \mu(s,f) ds  \leq \| f \|_{M_\psi} \cdot \int_{0}^{t} \psi'(s) ds \leq \| f \|_{M_\psi} \cdot \int_{0}^{t} \mu(s,\psi') ds \leq \int_{0}^{t} \mu(s,\psi') ds$$
for all $f \in K$ and $t \in (0,1),$ i.e. $|f| \prec \prec \psi'$ for all $f\in K.$ Thus, the assertion (b) holds with $g=\psi'$.

 $ \eqref{b} \Longrightarrow \eqref{a}$.
Suppose there is a positive function $g \in L_1(0,1)$ such that $|f| \prec \prec g$ for all $f \in K.$ Then by Lemma \ref{N-function} (see also \cite[Chapter II, p.60]{KR}) there exists an $N-$function $G$ (hence an Orlicz function) with $\frac{G(t)}{t} \to \infty$ as $t \to \infty$ such that $\int_{0}^{1} G(g(s))ds < \infty.$ In other words, $g \in L_G(0,1)$. We have $\int_{0}^{t} |f(s)|ds \leq \int_{0}^{t} \mu(s,f)ds$ for all $t \in (0,1]$ (see e.g. \cite[(2.12), p.64]{KPS}). By the assumption,
we have $\int_{0}^{t} \mu(s,f)ds \leq \int_{0}^{t} \mu(s,g)ds$  for all $f \in K$ and for all $t \in (0,1]$ and so, by Lemma \ref{Prop:lefunction2} and  \cite[Lemma 2.5 (iv)]{FK}, we have

$$\int_{0}^{t} G(|f(s)|)ds \leq
 \int_{0}^{t} G(\mu(s,f))ds \leq \int_{0}^{t} G(\mu(s,g))ds < \infty.
 %,\ \forall f \in K\ {\rm and}\\forall t \in (0,1].
$$
This completes the proof.
\end{proof}

\section{Equivalence of Chong's and De la Vall\'{e}e Poussin's criteria in $L_1(0,\infty)$}\label{section semifinite}

In this section, we prove the equivalence of Chong's and De la Vall\'{e}e Poussin's criteria in $L_1$ over a $(0,\infty)$ under natural additional condition.

\begin{remark}
Recall that the classical Dunford's criterion identifies bounded and uniformly integrable subsets of $L_1(I)$ (where $m(I)< \infty$) with relatively weakly compact sets (\cite[Theorem 15, p.76]{Diestel}, \cite[Theorem 23, p.20]{Meyer}). Note, however, that this criterion of weak compactness is no longer valid in $L_1(0, \infty)$ as the following example illustrates.

Let $M=\{f_n=\frac{1}{n} \chi_{[n.2n]}\}_{n=1}^{\infty}$. Clearly, $M$ is norm bounded in $L_1(0, \infty)$ and uniformly integrable.
%
%
%Indeed, let $\varepsilon>0$ be given, then we choose $\delta=\varepsilon$. Hence,
%$$\sup_{n \geq 1} \int_{E} \frac{1}{n}\chi_{[n, 2n]}=\sup_{n \geq 1} \frac{1}{n} \cdot m([n,2n]\cap E) \leq \sup_{n \geq 1}  m([n,2n]\cap E) < \varepsilon.$$
%
However, $M$ is not relatively weakly compact  in $L_1(0, \infty)$.
\end{remark}

\begin{remark}
Neither De la Vall\'{e}e Poussin's criterion (condition \eqref{a} in Theorem \ref{ChVP}), nor Chong's criterion (condition \eqref{b} in Theorem \ref{ChVP}) describe relatively  weakly compact subsets in $L_1(0, \infty)$.
%We illustrate this statement in the following example.

%\begin{example}
For example, let $K=\{f_n=\chi_{[n, n+1]}\}_{n=0}^{\infty}$. Obviously, $K$ is a bounded subset of $L_1(0, \infty)$, which is not relatively weakly compact in $L_1(0, \infty)$. However,  $|f_n| \prec \prec g$ for all $f_n \in K$, where $g(x)=\chi_{[0,1]}(x)+\frac{1}{x^\alpha} \chi_{(1, \infty)}(x)$, where $\alpha >1$.

Also,  taking $G(x)=x^\alpha$,  $\alpha >1$, we obtain an Orlicz function $G$ with $\frac{G(t)}{t} \to \infty$ as $t \to \infty$ such that
$$\sup \left\{ \int_{0}^1 G(|f|)ds: \ f\in K \right\}<\infty.$$

\end{remark}

The following theorem is an extension of Theorem \ref{ChVP} to a $\sigma$-finite measure space, provided some additional condition on a set $K$.
\begin{theorem}\label{expl of th}
Let $K$ be a bounded subset of $L_1(0,\infty)$ satisfying the following condition:
\begin{equation}\label{1}
\sup_{f\in K} \int_{N}^{\infty} \mu(s,f) ds \to 0, \quad N \to \infty.
\end{equation}
Then the following two conditions are equivalent:
\begin{enumerate}[{\rm (i)}]
\item there exists an Orlicz function $G$ with $\frac{G(t)}{t} \to \infty$ as $t \to \infty$ so that $$\sup \left\{ \int_{0}^{\infty} G(|f|)ds: \ f\in K \right\}<\infty;$$
\item there exists a positive function $g\in L_1 (0,\infty)$ such that $|f| \prec\prec g$ for all $f \in K.$
\end{enumerate}
\end{theorem}

\begin{proof}
A quick analysis of the proof of the implication $\eqref{b} \Longrightarrow \eqref{a}$ in Theorem \ref{ChVP} shows that it holds verbatim for bounded subsets $K$ in $L_1(0, \infty)$ (even without the condition \eqref{1}). 

Now, we show that the implication $\eqref{a} \Longrightarrow \eqref{b} $ holds under an additional assumption \eqref{1}. Define a concave function $\psi$ on $(0,\infty)$ analogously as in the proof of the Theorem  \ref{ChVP}. Hence, we have $\int_{0}^{t} \mu(s,f) ds  \leq \int_{0}^{t} \mu(s,\psi')ds$ for all $f \in K$ and $t\in(0,\infty)$.

Fix $\varepsilon>0$. Due to (\ref{1}) there exists a real number  $N>1$ such that
$$\sup_{f\in K} \int_{N}^{\infty} \mu(s,f) ds < \varepsilon.$$
We define
\[
  g(s) :=
  \begin{cases}
                                   \psi'(s)+\varepsilon & \text{if $0 \leq s \leq N$}, \\
  1/{s^\alpha} & \text{if $s>N$,}
  \end{cases}
\]
where $\alpha>1$. Clearly, $g \in L_1(0, \infty)$ is  a positive function and $|f| \prec \prec g$ for all $f \in K,$ which completes the proof.
\end{proof}

\section{Appendix}

The following lemma is, most probably, well known. However, since we could not find any suitable reference, we include its proof here for the sake of convenience.

\begin{lemma}\label{sequence}
Let $\{x_n\}_{n=1}^{\infty}$ be a sequence of real numbers such that the series $\sum_{n=1}^{\infty}|x_n|$ is convergent. Then there exists a sequence of real numbers $\{y_n\}_{n=1}^{\infty}$ such that $\lim_{n\to \infty} y_n=\infty$ and the series $\sum_{n=1}^{\infty}|x_n y_n|$ is convergent.
\end{lemma}
\begin{proof}
Let us construct a (strictly) increasing sequence of natural numbers $\{n_l\}_{l=1}^{\infty}$ as follows. By the Cauchy's theorem we can find $n_1\in \mathbb{N}$ such that
$$\sum_{k=n_1}^{n}|x_k|<1, \ \ \ \text{for any} \ \ n>n_1.$$
Similarly, we can find $n_2>n_1$ such that
$$\sum_{k=n_2}^{n}|x_k|<\frac{1}{2}, \ \ \ \text{for any} \ \ n>n_2.$$
Continuing this procedure we construct the sequence $\{n_l\}_{l=1}^{\infty}$ such that $n_{l+1}>n_l$ for all $l \in \mathbb{N},$ and

\begin{equation}\label{sum}
\sum_{k=n_l}^{n}|x_k|<\frac{1}{2^{l-1}}, \ \ \ \text{for any} \ \ n>n_l.
\end{equation}
%and for any $l \geq 1.$

Now we construct a nondecreasing sequence $\{y_n\}_{n=1}^{\infty}$ such that $\lim_{n\to \infty} y_n=\infty$. Put $y_n=1$ for any $1\leq n \leq n_1$ and

$$y_n=l-1 \ \ \ \text{for any} \ \ \ n_{l-1}< n \leq n_l, \ \ l \geq 2.$$

It is easy to see that $\{y_n\}_{n=1}^{\infty}$ is nondecreasing. Moreover, 
$$\lim_{n\to \infty} y_n=\sup_{n\in \mathbb{N}}y_n \geq \sup_{l \geq 1} y_{n_l}=\infty.$$

Finally, we show that the series $\sum_{n=1}^{\infty}|x_n y_n|$ is convergent by using the Cauchy's theorem.

Let $\varepsilon>0.$ Since the series $\sum_{k=1}^{\infty}\frac{k}{2^k}$ is convergent we can choose $l_0=l_0(\varepsilon) \in \mathbb{N}$ such that $\sum_{k=l_0}^{\infty}\frac{k}{2^k} < \varepsilon.$

Let $n \in \mathbb{N}$ be such that $n>n_{l_0}.$ Let $m>n,$ consider the sum

$$\sum_{k=n}^{m}|x_k y_k|=\sum_{k=n}^{m}|x_k| y_k.$$

Define $s>l_0$ by condition $n_{s-1}<m\leq n_s.$ We have

$$\sum_{k=n}^{m}|x_k| y_k \leq \sum_{k=n_{l_0}+1}^{n_s}|x_k| y_k=\sum_{i=l_0}^{s-1}\sum_{k={n_i+1}}^{n_{i+1}}|x_k| y_k.$$

Since $y_k=i$ for any $n_i<k \leq n_{i+1},$ we obtain

$$\sum_{i=l_0}^{s-1}\sum_{k={n_i+1}}^{n_{i+1}}|x_k| y_k=\sum_{i=l_0}^{s-1} i\sum_{k={n_i+1}}^{n_{i+1}}|x_k|.$$

By the definition of the sequence $\{n_i\}$ and inequality (\ref{sum}), we have

$$\sum_{i=l_0}^{s-1} i\sum_{k={n_i+1}}^{n_{i+1}}|x_k| \leq \sum_{i=l_0}^{s-1} \frac{i}{2^{i-1}} \leq \sum_{i=l_0}^{\infty} \frac{i}{2^{i-1}} \leq 2\varepsilon.$$

Therefore, for any $\varepsilon>0$ there exists $n_0=n_0(\varepsilon)=n_{l_0}$ such that for any $n>n_0$ and any $m>n$

$$\sum_{k=n}^{m}|x_k y_k| \leq 2\varepsilon,$$
which completes the proof.
\end{proof}

\section{Acknowledgment}
This work is dedicated to celebrating the 60th anniversary of Fedor Sukochev. In fact,
the present paper was written in collaboration with Fedor Sukochev several years ago. Many of our research papers stand as a testament to his invaluable mentorship and support, enabling numerous research endeavors. The authors also would like to thank A.A. Sedaev, Y.M. Semenov, J. Huang and T. Scheckter for helpful discussions, and thank A. Mukanov for his assistance in proving Lemma \ref{sequence}.

The authors were supported by the Science Committee of the Ministry of Science and Higher Education of the Republic of Kazakhstan (Grant No. AP14869301).

\end{document}